\documentclass[12pt]{article}
 \usepackage{amsmath,amssymb,amsthm} 
\usepackage{url} 

\title{Legendre-Teege Quadratic Reciprocity}

\author{Mark B. Villarino\\
Escuela de Matem\'atica, Universidad de Costa Rica,\\
10101 San Jos\'e, Costa Rica}

\date{}

\theoremstyle{plain}
\newtheorem{thm}{Theorem}[section]   

\newtheorem{cor}[thm]{Corollary}     
\newtheorem{lemma}[thm]{Lemma}       
\theoremstyle{definition}
\newtheorem{defn}{Definition}[section] 
\theoremstyle{definition}
\newtheorem{exmp}{Example}[section]

\numberwithin{equation}{section}

\makeatletter
\def\section{\@startsection{section}{1}{\z@}{-3.5ex plus -1ex minus
			  -.2ex}{2.3ex plus .2ex}{\large\bf}}
\def\subsection{\@startsection{subsection}{2}{\z@}{-3.25ex plus -1ex
			  minus -.2ex}{1.5ex plus .2ex}{\normalsize\bf}}
\makeatother

\renewcommand{\leq}{\leqslant}  

\begin{document}

\maketitle


\section{Introduction} 
\label{sec:intro}

The law of quadratic reciprocity is one of the most celebrated theorems in all of mathematics.  Its fame rests not only on its technical use in the theory of numbers but also on the influence it has had in the creation and development of entire branches of mathematics such as commutative algebra and elliptic curves.  Franz Lemmermeyer's important treatise \emph{Reciprocity Laws}~\cite{Lemmermeyer} details the history and mathematics of its proofs and is well worth perusing.

The quadratic reciprocity law also enjoys another extraordinary characteristic. \emph{ More proofs of it have been published than of any other mathematical theorem, except for Pythagoras' theorem. } Indeed, no other theorem (famous or not) even comes close.  Lemmermeyer ~\cite{Baumgart} cites 314 published proofs up to 2015 and most likely a few more have appeared since.

So it is an historical irony and curiosity that the \emph{first attempted} published proof is \emph{not} on that list.

We briefly review its history.

The famous French mathematician, Adrien-Marie Legendre (1752-1833), published a treatise on the theory of numbers~\cite{Legendre} in which he:
\begin{itemize}
  \item introduced the ``Legendre symbol";
  \item introduced the word ``reciprocity;"
  \item stated the quadratic reciprocity theorem in the form given in all texts today;
  \item offered a proof that contained a ``gap."
\end{itemize}

His proof was a revision of an earlier (1785) proof but it had a new gap.

We formulate this new gap in the form of a lemma.

\begin{lemma}(Legendre's Lemma)
\label{LL}
Given any prime of the form $a:=4n+1$ there exists a prime of the form $4n+3$ of which $a$ is a quadratic non-residue.
\end{lemma}

For example $a=17$ is a quadratic nonresidue of $b=3$.

Legendre showed that he could complete his proof if he could prove his lemma.  Although he made several attempts he never succeeded.  Indeed, it appears that the lemma is much \emph{harder} to prove than the theorem, itself.  

In 2011, Weintraub \cite{Weintraub} published in this Monthly a very nice and detailed account of Legendre's attempted proofs of the quadratic reciprocity law , but only cited, without proof, Legendre's unproved lemma as ``Hypothesis B."

In the subsequent 220+ years after Legendre only two proofs of the lemma have appeared.  One is by a relatively unknown German mathematician, Hermann Teege~\cite{Teege} in 1923, while the other was authored by Kenneth Rogers~\cite{Rogers} and was published in 1971.  We will say more about these proofs later on.

The literature does not seem to contain an easily accessible, ab initio,  connected, detailed, rigorous and complete presentation of a proof of Legendre's Lemma.  This paper is designed fill that gap.

In order to make our paper as self-contained as possible we will first outline the major steps in Legendre's original proof in its final form and then fill the gap.  He based his proof on\begin{itemize}
  \item  a theorem 
on quadratic forms which he rigorously proved; and
  \item  the theory of the Pell equation; and
    \item Legendre's (unproved) lemma; and
    \item the multiplicativity of the legendre symbol and for any prime $p$, Euler's result $ \left(\frac{-1}{p}\right)=(-1)^{\frac{p-1}{2}}.$
\end{itemize}

We will follow (and expand) Teege's arrangement \cite{Teege1} of Legendre's proof.


\section{Legendre's theorem on quadratic forms}

The theorem proven by Legendre~\cite{Legendre} \S27, and which is today rather famous in its own right, is the following:

\begin{thm}
\label{Legendrethm}
If $a,b,c$ denote numbers which are relatively prime in pairs, not one of which is zero nor divisible by a square, then the equation
\begin{equation}
\label{LE}
ax^2+by^2+cz^2=0
\end{equation}

\noindent has no solution in integers unless $-bc,-ac,-ab$ are, respectively, quadratic residues of $a,b,c$, and at least one of them is positive and one of them is negative.  However, if these four conditions are fulfilled, then the equation is solvable in integers.

\end{thm}
$\hfill \blacksquare$

\begin{cor} 
\label{Legendrecor}
Using the Legendre symbol, if \eqref{LE} has no solutions, and if $a$, $b$, $c$ be primes or $1$, then at least one of the following equations must hold:
$$
\left(-\frac{bc}{a}\right)=-1\quad \text{or},\left(-\frac{ca}{b}\right)=-1 \quad\text{or},\left(-\frac{ab}{c}\right)=-1 .
$$

\end{cor}

$\hfill \blacksquare$

Legendre's proof of the quadratic reciprocity law also uses the following simple lemma .
\begin{lemma}
\label{Legendrecor2}
 If the integers $a,b,c$ are all $\equiv 1\pmod 4$, then the equation
$$ax^2+by^2+cz^2=0$$
is not solvable in integers.  
\end{lemma}
\begin{proof}
For, since the square of an odd number is of the form $8n+1$, it is easy to see in this case that the value of $ax^2+by^2+cz^2$ is necessarily either $\equiv 1$, or $\equiv 2$, or $\equiv 3\pmod 4$ unless all the quantities $x,y,z$ are simultaneously even.  But this last possibility is absurd since if we divide the equation by $2^{2\mu}$ where $2^{\mu}$ is the highest power of $2$ dividing $x,y,z$ simultaneously, the new solution $x',y',z'$ will have at least one odd member.
\end{proof}

\section{Pell's equation}

Legendre also uses the following consequence of the theory of Pell's equation~\cite{Legendre} \S49.

\begin{thm}
\label{Pell}
If  $b$, $B$ and $\beta$ are prime numbers of the form $4n+3$ and $a$ is a prime number of the form $4n+1$, then it is always possible to solve one of the six equations
$$aM^2-b\beta N^2=\pm 1,$$
$$bM^2-a\beta N^2=\pm 1,$$
$$\beta M^2-ab N^2=\pm 1.$$
\end{thm}
\begin{proof}

It is well known that for every arbitrary positive integer $D$ the equation
$$x^2-Dy^2=1$$
is always solvable in distinct non zero integers $x,y$.  Now, if $T$ and $U$ are the smallest positive solutions among them then
$$T^2-DU^2=1$$
or
$$(T-1)(T+1)=DU^2.$$
Moroever if $D\equiv 1\pmod 4$, it is easy to see that $T$ must be odd and $U$ even so that we can put
$$\frac{T+1}{2}\cdot\frac{T-1}{2}=DU_1^2$$
where $U=2U_1$ and the two factors on the left must be relatively prime since their difference is $1$.

Now, if $D=b\cdot B$ where $b$ and $B$ are two prime numbers of the form $4n+3$, and if the factorization of $U_1$ into two relatively prime factors is given generally by
$$U_1=M\cdot N,$$
then it follows from the equation
$$\frac{T+1}{2}\cdot\frac{T-1}{2}=b\cdot B\cdot U_1^2$$
that all possible factorizations are given by the following four cases:
\begin{align*}
\label{}
    1.\quad \frac{T+1}{2}& =M^2&\frac{T-1}{2}&=b\cdot B\cdot N^2  \\
    2.\quad \frac{T+1}{2}& =b\cdot B\cdot M^2&\frac{T-1}{2}&= N^2  \\
    3.\quad \frac{T+1}{2}& =b\cdot M^2&\frac{T-1}{2}&= B\cdot N^2  \\
    4.\quad \frac{T+1}{2}& =B\cdot M^2&\frac{T-1}{2}&=b\cdot N^2  \\
\end{align*}

Of these Case $1$ is impossible because otherwise
$$M^2-b\cdot B\cdot N^2=1$$
which contradicts the minimality of the solution $T$ and $U$.

Similarly, Case $2$ is impossible because otherwise
$$1=b\cdot BM^2-N^2$$
and $-1$ would be a quadratic residue of a prime number of the form $4n+3.$

Two cases still remain:
$$1=bM^2-B\cdot N^2$$
and
$$1=BM^2-b\cdot N^2$$
one of which must occur.

\end{proof}


\section{Legendre's proof of the law of quadratic reciprocity}

After these preparations we can now briefly sketch Legendre's proof of the law of quadratic reciprocity.  

If we retain the earlier notation, so that $a,A$ are prime numbers of the form $4n+1$ and $b,B$ are prime numbers of the form $4n+3$, then the theorem splits into the following eight cases, where $ \left(\frac{a}{b}\right)$ is the well known Legendre symbol ~\cite{Legendre} \S166.

\begin{thm}
\label{QRT}
\begin{align*}
\label{}
    I.& \left(\frac{a}{b}\right)=-1&\Rightarrow&&  \left(\frac{b}{a}\right)=-1\\
    II.& \left(\frac{b}{a}\right)=+1&\Rightarrow&&  \left(\frac{a}{b}\right)=+1\\
     III.& \left(\frac{B}{b}\right)=+1&\Rightarrow&&  \left(\frac{b}{B}\right)=-1\\
      IV.& \left(\frac{B}{b}\right)=-1&\Rightarrow&&  \left(\frac{b}{B}\right)=+1\\
       V.& \left(\frac{a}{A}\right)=+1&\Rightarrow&&  \left(\frac{A}{a}\right)=+1\\
        VI.& \left(\frac{a}{A}\right)=-1&\Rightarrow&&  \left(\frac{A}{a}\right)=-1\\
         VII.& \left(\frac{a}{b}\right)=+1&\Rightarrow&&  \left(\frac{b}{a}\right)=+1\\
         VIII.& \left(\frac{b}{a}\right)=-1&\Rightarrow&&  \left(\frac{a}{b}\right)=-1\\
    &  
\end{align*}
\end{thm}

\begin{proof}

The idea of Legendre's proof is the following.  One forms a suitable equation \eqref{LE} in which the coefficients  are all $\equiv 1\pmod 4$,  which, by  lemma~\ref{Legendrecor2}, \emph{cannot} be solved, and uses Corollary~\ref{Legendrecor}  and the properties of the Legendre symbol to show that the left hand side of the implication (I)-(VIII) must therefore hold.
\\
\\
\fbox{\emph{Cases $I$ and $II$:}}
 \\
 \\
 We proceed from the \emph{impossibility} of the equation
$$x^2+ay^2-bz^2=0$$
Lemma 2.3 implies that we may not have, simultaneously,
$$ \left(\frac{-a}{b}\right)=+1\quad\text{and}\quad  \left(\frac{b}{a}\right)=+1.$$
Therefore, if
$$ \left(\frac{-a}{b}\right)=+1\quad\text{that is}\quad \left(\frac{a}{b}\right)=-1,$$
since the Legendre symbol can only have the values$+1$ and $-1$, it must follow that $ \left(\frac{b}{a}\right)=-1$, and if
$$ \left(\frac{b}{a}\right)=+1,$$
it must follow that $ \left(\frac{a}{b}\right)=+1.$
\\
\\
\fbox{\emph{Cases $III$ and $IV$:}}
 \\
 \\
 By the Theorem~\ref{Pell}, one of the two equations
$$bM^2-Bn^2=1$$
$$BM^2-bN^2=1$$
is satisfied.

However, if $ \left(\frac{B}{b}\right)=+1$, the equation $bM^2-BN^2=1$ cannot hold since otherwise $ \left(\frac{-B}{b}\right)=+1$ or  $\left(\frac{B}{b}\right)=-1$ contrary to the hypothesis.  Therefore the second equation holds, from which it follows that $ \left(\frac{-b}{B}\right)=+1$ or $ \left(\frac{b}{B}\right)=-1.$

On the other hand, if $ \left(\frac{B}{b}\right)=-1$, then we must exclude the second equation and the now valid first equation implies $\left(\frac{b}{B}\right)=+1.$
\\
\\
\fbox{\emph{Cases $V$ and $VI$:}}
\\
\\
\textbf{\emph{ Choose a prime number $\beta$ of the form $4n+3$ such that $\left(\frac{a}{\beta}\right)=-1.$}}  Then, by case $I$ we must have $\left(\frac{\beta}{a}\right)=-1.$  If we now consider the unsolvable equation
$$x^2+ay^2-A\beta z^2=0$$
we immediately see that we cannot have simultaneously that $(-a)$ is a residue of $(A\beta)$ and $(A\beta)$ a residue of $a$.

However, the first condition is fulfilled in case $V$ because then $ \left(\frac{-a}{A}\right)=+1$ and by the assumption about $\beta$ also $ \left(\frac{-a}{\beta}\right)=+1$ so that $(-a)$ must be a residue of $(A\beta)$.  Accordingly the second condition cannot also be fulfilled, i.e., we must have $ \left(\frac{\beta A}{a}\right)=-1$ or, since $\left(\frac{\beta}{a}\right)=-1$, also $ \left(\frac{A}{a}\right)=+1.$

Case $VI$ can be reduced to Case $V$.  For, if we could deduce the condition $ \left(\frac{A}{a}\right)=+1$ from the equation $ \left(\frac{a}{A}\right)=-1$, then the case just proved would imply $ \left(\frac{a}{A}\right)=+1$, contrary to the hypothesis.
\\
\\
\fbox{\emph{Cases $VII$ and $VIII$:,} }
\\
\\
Here therefore $ \left(\frac{a}{b}\right)=+1,$ we again as above \textbf{\emph{choose a prime number $\beta$ of the form $4n+3$ so that $ \left(\frac{a}{\beta}\right)=-1$ }}and therefore also
$ \left(\frac{\beta}{a}\right)=-1.$  Then, by the theory of the Pell equation, one can assume that one of the following six equations is fulfilled
$$\pm 1=aM^2-b\beta N^2,$$
$$\pm 1=bM^2-a\beta N^2,$$
$$\pm 1=\beta M^2-ab N^2.$$
for a suitable choice of sign on the left hand side.

Recalling that 
$$\left(\frac{a}{b}\right)=+1,\left(\frac{a}{\beta}\right)=-1,\left(\frac{\beta}{a}\right)=-1,$$
we conclude that the following equations are impossible
$$1.\quad +1=\beta M^2-ab N^2\quad\text{which assumes}\quad \left(\frac{\beta}{a}\right)=+1, $$
$$2.\quad -1=\beta M^2-ab N^2\quad\text{which assumes}\quad \left(\frac{\beta}{a}\right)=+1, $$
$$3.\quad -1=aM^2-b\beta N^2\quad\text{which assumes}\quad \left(\frac{a}{b}\right)=-1, $$
$$4.\quad +1=aM^2-b\beta N^2\quad\text{which assumes}\quad \left(\frac{a}{\beta}\right)=+1.$$

Therefore, only two equations remain
$$+1=bM^2-a\beta N^2$$
$$-1=bM^2-a\beta N^2$$
each of which, whichever one holds, requires the relation $\left(\frac{b}{a}\right)=+1.$  This proves Case (VII).

If  $ \left(\frac{b}{a}\right)=-1,$ then we must have $ \left(\frac{a}{b}\right)=-1,$ since if $ \left(\frac{a}{b}\right)=+1,$ we would have $ \left(\frac{b}{a}\right)=+1,$  which is a contradiction.  This proves Case (VIII).
\end{proof}


\section{The gap in the proof}

Again we state the fundamental gap in Legendre's proof, already mentioned in the introduction, and which yet has to be filled.  Namely, that \emph{given a prime number $a$ of the form $4n+1$ it is always possible to find a prime number $\beta$ of the form $4n+3$, of which $a$ is a quadratic nonresidue. } It is easy to prove the existence of this auxiliary prime for prime numbers of the form $8n+5$ since then $1+a$ is of the form $8n+6$, therefore divisible by a prime of the form $4n+3$, of which $a$ is then $a$ is a nonresidue.  But the proof for the case $a=8n+1$ is profoundly more difficult.  

Some forty years after Legendre,  Dirichlet  introduced brilliant new methods into number theory in his famous determination of the \emph{class number }of binary quadratic forms.  On the basis of these investigations one \emph{can} provide the necessary proof to complete Legendre's demonstration.

Nevertheless, almost 125 years passed until (1923) when finally the German mathematician H. Teege posthumously published the first rigorous proof of Legendre's lemma~\cite{Teege}.  He based his proof on the (easy part (!)) of the Dirichlet class number formulas mentioned above, and was able to obtain the following much stronger result.
\begin{thm} (Teege's Theorem)
\label{Teege}
It is always possible to find not only one, but infinitely many prime numbers $\beta$ of the form $4n+3$ so that a given prime number $a$ of the form $8n+1$ is a quadratic nonresidue of them, and that there is always at least one of them smaller than $a$.
\end{thm}  

We note that Gauss' first proof of the quadratic reciprocity law proved the existence of such a prime number $\beta$ but \emph{without} the condition that it be of the form $4n+3.$

Then, some 48 years later (1971) K. Rogers published a proof~\cite{Rogers} based on a lemma of Selberg that the latter used in his elementary proof of the prime number theorem for arithmetic progressions~\cite{Selberg}.  Just like Teege, Rogers also proved that there are \emph{infinitely many} primes $\beta$ satisfying Legendre's lemma, but not the fact that one can find such a prime  smaller than $a$.

These, apparently, are the only two proofs of Legendre's lemma in the literature.

We will elaborate a version of Teege's proof and sketch the principle part of Rogers' proof.
\section{Binary Quadratic Forms}

We briefly review some facts about binary quadratic forms.  See Landau \cite{Landau} and Mathews \cite{Mathews} for more details.

\begin{defn} A \textbf{binary quadratic form} (BQF) $Q(x,y)$ is a function
$$Q(x,y):=ax^2+bxy+cy^2$$
where $a,b,c$ are integers and will be denoted $(a,b,c)$.  The quantity
$$d:=b^2-4ac$$
is the \textbf{discriminant} and $D:=\dfrac{d}{4}$ is the \textbf{determinant.} The discriminant will always denote a nonsquare integer congruent to $0$ or $1$ modulo 4. It is called a \textbf{fundamental} discriminant.  If $d$ is divisible by $4$ then $b$ is even.
\end{defn}

\begin{defn} We say that a BQF $Q$ \textbf{represents} a number $n$ when there exist integers $x,y$ such that $Q(x,y)=n$.  If $x$ and $y$ are relatively prime, we say that $Q$
\textbf{properly} represents $n$.
\end{defn}

A famous theorem of Lagrange states the following.
\begin{thm}
The forms of given (fundamental) discriminant $d$ fall into classes of mutually \textbf{equivalent} forms under linear substitutions of the form
\begin{equation}
\label{equiv}
x=\alpha x'+\beta y', y=\gamma x'+\delta y'
\end{equation}
with integral coefficients $\alpha,\beta,\gamma,\delta$ satisfying $\alpha\delta-\beta\gamma=1$.
The number of classes, $h(d)$ (called the \textbf{class number}), for a given discriminant, $d$, is \textbf{finite}.
\end{thm}
$\hfill \blacksquare$

We give examples below.


\section{Teege's Identity}

We can now state \textbf{\emph{Teege's identity}} which is the basis of his proof and a byproduct of Dirichlet's researches.
\begin{thm} (Teege's Identity)
\label{TeegeID}
Let $D>2$ be a nonsquare positive integer and $p$ a prime which does not divide $D$.  Then
$$
\boxed{\frac{h(D)\ln(T+U\sqrt{D})}{h(-D)\cdot 2\pi}=\lim_{s\rightarrow 1}\dfrac{\displaystyle\prod_{ \left(\frac{D}{p}\right)=1}\frac{1+\frac{1}{p^s}}{1-\frac{1}{p^s}}}{2\displaystyle\prod_{ \left(\frac{-D}{p}\right)=1}\frac{1+\frac{1}{p^s}}{1-\frac{1}{p^s}}}}
$$
where $h(D)$ and $h(-D)$ are the class numbers of binary quadratic forms of determinant $\pm D$,  $T,U$ is the minimum positive solution of the Pell equation $T^2-DU^2=1$, and $\left(\dfrac{\pm D}{p}\right)$ is the Legendre symbol.

\end{thm}

\begin{exmp}
Let $d=\pm 20$.  Then $\pm D=\pm 5$.  It is well-known that the primes $p$ with $\left(\frac{5}{p}\right)=1$ are of the two forms $p=5m\pm 1$ and the primes $p$ with  $\left(\frac{-5}{p}\right)=1$ 
are of the four forms $20n+1,20n+9,20n+3,20n+7$.  Moreover (see the next two examples) $h(5)=1, h(-5)=2$.  Thus, Teege's identity takes the form
$$
\boxed{\frac{\ln(9+4\sqrt{5})}{2\cdot 2\pi}=\lim_{s\rightarrow 1}\dfrac{\displaystyle\prod_{ p=5m\pm 1}\frac{1+\frac{1}{p^s}}{1-\frac{1}{p^s}}}{2\displaystyle\prod_{ p=20n+1,20n+9,20n+3,20n+7}\frac{1+\frac{1}{p^s}}{1-\frac{1}{p^s}}}}
$$
that is
$$
\boxed{\frac{\ln(9+4\sqrt{5})}{2\pi}=\frac{\frac{1+\frac{1}{11}}{1-\frac{1}{11}}}{\frac{1+\frac{1}{3}}{1-\frac{1}{3}}}\cdot\frac{\frac{1+\frac{1}{19}}{1-\frac{1}{19}}}{\frac{1+\frac{1}{7}}{1-\frac{1}{7}}}\cdot\frac{\frac{1+\frac{1}{29}}{1-\frac{1}{29}}}{\frac{1+\frac{1}{29}}{1-\frac{1}{29}}}\cdot\frac{\frac{1+\frac{1}{31}}{1-\frac{1}{31}}}{\frac{1+\frac{1}{37}}{1-\frac{1}{37}}}\cdots}
$$
\
\end{exmp}

The idea of Teege's proof is this. We argue by contradiction.  We let $D$\emph{ be a prime for which Legendre's lemma is false}.  Then the left hand side of Teege's identity is a\emph{ finite number} while the right hand side is shown to \emph{diverge to infinity}.  Thus Legendre's lemma is not false for any prime, that is,\emph{ it is always true. } The hard part of the proof of the identity is the left-hand side.  Dirichlet showed that the numerator and denominator are measures of the average number of distinct representations of a number $m$ by binary quadratic forms of determinant $\pm D$, respectively, as $m\rightarrow \infty$.

It turns out that Teege's identity is a consequence of two other results: the first  we call the \emph{\textbf{fundamental identity}} which we state later on after some introductory examples; and the second is Dirichlet's \emph{\textbf{class number formula.}}

Then we will devote the rest of our paper to proving the identity and discussing the class number formulas, and apply them to prove Teege's identity.  Finally we'll use Teege's identity to give a detailed proof of Legendre's Lemma.  We will follow with a discussion of Rogers' proof.


\section{Dirichlet's ``fundamental identity"}
  
  We begin with two examples

\begin{exmp}
Let $d:=-20$.  Then it is well known that any BQF with $d:=-20$ is equivalent to one of the two BFQ's:
$$
Q_1(x,y):=x^2+5y^2,\quad Q_2:=2x^2+2xy+3y^2.
$$
Thus, the class number
$$
h(-20)=2.
$$

The numbers properly and improperly represented by $Q_1$ include
$$
0, 1, 4, 5, 6, 9, 14, 16, 20, 21, 24, 25, 29, 30, \cdots
$$
 while the numbers properly and improperly represented by $Q_2$ include
 $$
 0, 2, 3, 7, 8, 10, 12, 15, 18, 23, 27, 28, 32, 35,\cdots
 $$
 
 We will be interested in the subset of those integers represented by $Q_1$ and $Q_2$ which are \emph{relatively prime} to $d:=-20$, i.e. to
 $$
 1,3,7,9,21,23,27,29,41,43,47,49,61,63,67,69,81,\cdots.
 $$

 We will  study the series (today called an \emph{Epstein zeta function})
 $$
 \sum_{x,y}\frac{1}{Q_1(x,y)^s}+\sum_{x,y}\frac{1}{Q_2(x,y)^s}=\frac{2}{1^s}+\frac{4}{3^s}+\frac{4}{7^s}+\frac{4}{9^s}+\frac{8}{21^s}+\cdots
 $$
 where $Q$ runs over the two BFQ's $Q_1$ and $Q_2$ once and $x,y$ run over all relatively prime pairs with $Q(x,y)$ relatively prime to $d$ , while the numerator is t\emph{he number of proper representations of the denominator by some $Q(x,y)$ of determinant $-5$}.  For example, $m=9$ is represented properly by $Q_1$ by the four pairs $(x,y)=(2,1),(-2,1),(2,-1),(-2,-1)$ and $m=7$ is represented properly by $Q_2$ by the four pairs $(x,y)=(1,1),(-1,-1),(-1,2),(1,-2).$

\end{exmp}

\begin{exmp}
Let $d:=+20$.  Then it can be proved that any BQF with $d:=+20$ is equivalent to the \emph{single} BQF
$$
Q(x,y):=x^2-5y^2.
$$
Therefore
$$
h(+20)=1.
$$
The numbers properly represented by $Q(x,y)$ which are prime to $d:=+20$ are
$$1,9,11,19,29,31,41,49,59,61,71,79,81,89,99,\cdots ,$$
and we will study the Epstein zeta series 
$$
\sum_{x,y}\frac{1}{Q(x,y)^s}=\frac{2}{1^s}+\frac{4}{9^s}+\frac{4}{11^s}+\frac{4}{19^s}+\cdots
$$
where again the numerator is the number of proper representations of the denominator by the one $Q(x,y)$ of determinant $+5$.
\end{exmp}

These two examples illustrate the left side of a fundamental identity which we now study.  


The first step in the proof of our final identity is the following standard counting result (see \cite{Mathews} \S202).

Let
$$
w:= \begin{cases} 6 \quad\text{if}\quad d=-3,\\ 4\quad\text{if}\quad d=-4,\\ 2\quad\text{if}\quad d<-4,\\ 1\quad\text{if}\quad d>0.
\end{cases}
$$
\begin{lemma}
\label{countinglemma}
Suppose $n$ is an integer relatively prime to $2D$.  If $n$ is divisible by $\mu$ distinct primes each of which satisfies $\left(\frac{D}{p}\right)=1$, but by no other primes, then $n$ can be represented in $w2^{\mu}$ distinct ways by a primitive form of determinant $D$.  Otherwise $n$ cannot be represented by a primitive form of determinant $D$.
\end{lemma}
$\hfill \blacksquare$
\\

For example if $Q(x,y)=3x^2+2xy+2y^2=7$ the theorem states that it has $2^{1+1}$ solutions, and indeed the pairs $(x,y)=(\pm 1,\pm 1),(\pm 2,\mp 1)$ are the four solutions predicted.

By applying the above lemma we obtain the following result.
\begin{thm}
The following identity is valid
$$
\sum_{Q}\sum_{x,y}\frac{1}{Q(x,y)^s}=w\sum_{m=1}^{\infty} \frac{2^{\mu}}{m^s} = w\prod_{ \left(\frac{D}{p}\right)=1}\frac{1+\frac{1}{p^s}}{1-\frac{1}{p^s}}
$$
where $Q$ runs over each class once and $x$ and $y$ run over all relatively prime pairs with $Q(x,y)$ relatively prime to $D$.

\end{thm}

\begin{proof}  
We have to prove the right-hand equality.
We follow Dirichlet( \cite{Dirichlet 2} pp. 154-155)
Let $$
p_1,p_2,p_2,\cdots 
$$
be the primes in the above product.  Then each $m$ in the above sum is of the form
$$
p_1^{n_1}p_2^{n_2}p_3^{n_3}\cdots
$$
where the exponents $n_1,n_2,n_3,\cdots$ are positive integers or zero, and each $m$ is uniquely expressible in this form.  If we now form infinite series corresponding to these primes
$$
1+\frac{2}{p_1^{s}}+\frac{2}{p_1^{2s}}+\frac{2}{p_1^{3s}}+\cdots\frac{2}{p_1^{n_1s}}+\cdots
$$
$$
1+\frac{2}{p_2^{s}}+\frac{2}{p_2^{2s}}+\frac{2}{p_2^{3s}}+\cdots\frac{2}{p_2^{n_2s}}+\cdots
$$
$$
1+\frac{2}{p_3^{s}}+\frac{2}{p_3^{2s}}+\frac{2}{p_3^{3s}}+\cdots\frac{2}{p_3^{n_3s}}+\cdots
$$
and so on, the product of arbitrary terms of the first, second, third series, etc., has the form
$$
\frac{2^{\mu}}{(p_1^{n_1}p_2^{n_2}p_3^{n_3}\cdots)^s}=\frac{2^{\mu}}{m^s}
$$
where $\mu$ is the number of primes $p$ actually dividing $m$, i.e., with non-zero exponent $n$.  Thus, by unique factorization,  we get each term in our sum exactly once.  

On the other hand
$$
1+\frac{2}{p^{s}}+\frac{2}{p^{2s}}+\frac{2}{p^{3s}}+\cdots\frac{2}{p^{ns}}+\cdots=\frac{2}{p^s}\frac{1}{2-\frac{1}{p^s}}=\frac{1+\frac{1}{p^s}}{1-\frac{1}{p^s}}
$$
This completes the proof.
\end{proof}

Now we transform the identity into a final form.  First we multiply both sides by $\displaystyle\sum_{(n,2D)=1}\frac{1}{n^{2s}}$.

Then the left-hand side of the modified identity is a triply-infinite series whose first term is of the form (say):
$$
\frac{1}{(an^2x^2+2bn^2xy+cn^2y^2)^{s}}.
$$
But, if we put
$$
x':=nx,y':=ny
$$
then the modified left-hand side becomes a doubly infinite series
$$
\sum\frac{1}{(ax'^2+2bx'y'+cy'^2)^{s}}
$$
where $x,y$ are no longer necessarily relatively prime.

We have now proven what Dirichlet called  (\cite{Dirichlet 2},\S 88) the \textbf{fundamental identity.}

\begin{thm}
\label{FI}
$$
\boxed{\sum_{Q}\sum_{x,y}\frac{1}{Q(x,y)^{s}}=k\displaystyle\sum_{(n,2D)=1}\frac{1}{n^{2s}}\cdot\prod_{ \left(\frac{D}{p}\right)=1}\frac{1+\frac{1}{p^{s}}}{1-\frac{1}{p^{s}}}}
$$
where $Q(x,y)$ is relatively prime to $2D$, and, if $D>0$, then
$$y>0, U\cdot(ax+by)>Ty$$
where $T,U$ are the smallest positive integer solutions of 
$T^2-DU^2=1.$  Finally, $x,y$ run over all pairs of positive integers where $Q(x,y)$ satisfies the above conditions.

\end{thm}
$\hfill \blacksquare$


\section{Dirichlet's class number formulas and Teege's identity.}

Let $d$ be the discriminant of the quadratic forms $Q(x,y)$ with class number $h(d)$.  The left hand side of the fundamental identity is an infinite series whose general term is 
$$
\frac{r_Q(m)}{m^{s}}
$$
where  $r_Q(m)$ is the number of pairs $(x,y)$ of integers which satisfy
$$
Q(x,y)=m.
$$
and $s>1$.
In 1839 (\cite{Dirichlet 1}) Dirichlet published a brilliant \emph{tour de force} of elementary (though not simple) number theory and calculus in which he  shows that if the fundamental identity is multiplied by $(s-1)$, then \emph{each side tends to a finite limite as} $s\rightarrow 1$.  Moreover, surprisingly, each summand on the left-hand side tends to the \emph{same} limit, which depends only on $D$.  Dirichlet evaluates this limit by \emph{geometric} arguments (area of an \emph{ellipse} when $D<0$, area of a \emph{hyperbolic sector} when $D>0$).  Since there are $h(D)$ terms on the left-hand side,  and when $s\rightarrow 1$ they all tend to the same limit, the \emph{left-hand side tends to a number of the form} 
$$
h(D)\times(\text{this geometric limit})
$$.

Unfortunately, space prevents us from giving the proofs but they can be found in Dirichlet's text (\S 88-\S 104 \cite{Dirichlet 2}).

We cite his final result.
\begin{thm} (Dirichlet's Class Number Formulas)
\label{CLN}
Let $s>1$.
If $D<0$ then
$$
\boxed{\lim_{s\rightarrow 1}(s-1)\cdot\sum_{Q}\sum_{x,y}\frac{1}{Q(x,y)^{s}}=\frac{h\cdot \pi}{\sqrt{|D|}}}
$$
If $D>0$ then
$$
\boxed{\lim_{s\rightarrow 1}(s-1)\cdot\sum_{Q}\sum_{x,y}\frac{1}{Q(x,y)^{s}}=\frac{h}{2\sqrt{D}}\ln(T+U\sqrt{D}).}
$$

\end{thm}
$\hfill \blacksquare$

The standard statement of the class number formulas, namely
$$
\frac{h\pi}{2\sqrt{|D|}}=\sum_{(n,2D)=1)}\frac{1}{n}\left(\frac{D}{n}\right),\quad \frac{h}{2\sqrt{D}}\ln(T+U\sqrt{D})=\sum_{(n,2D)=1)}\frac{1}{n}\left(\frac{D}{n}\right)
$$
where $\left(\frac{D}{n}\right)$ is the Jacobi symbol, has the infinite products in the fundamental identity replaced by ``$L$-series" which Dirichlet obtains by applying the law of quadratic reciprocity.  Of course, we had to avoid that.  Our paper has described only the ``easy" half of his full investigation.

Historically it is worth noting that Dirichlet investigated the class number because his proof of the infinitude of primes in an arithmetic progression required that he prove that a certain  $L$-series always have a non-zero value and it turns out that that series has a sum that is a non-zero multiple of a class number.

We also point out that Rogers' proof \cite{Rogers} of Legendre's Lemma cites a proposition in Selberg's proof \cite{Selberg} of the prime number theorem for arithmetic progressions and the proof of that proposition also counts lattice points in domains bounded by conic sections, as did Dirichlet above, a technique originated by Gauss.


\begin{proof} (Proof of Teege's Identity)

If we divide the second class number formula in Theorem~\ref{CLN} by the first and apply the fundamental identity Theorem~\ref{FI} we finally obtain \emph{\textbf{Teege's identity}} Theorem 3.2. 
\end{proof}


\section{Teege's proof of Legendre's lemma} 

Teege's proof has two steps:
\begin{itemize}
  \item first he transforms the fundamental identity by assuming $\pm D$ to be a prime of the form $4n+1$.  This allows for simplifying cancellations in the numerator and denominator of Teege's identity.
  \item then he further assumes that $p$ is a prime for which Legendre's lemma is false, and obtains a contradiction.
\end{itemize}
We write the \emph{fundmental identity} in the form
$$\boxed{\sum\frac{1}{Q(x,y)^{s}}=k\displaystyle\sum_{(n,2D)=1}\frac{1}{n^{2s}}\cdot\prod_{ \left(\frac{D}{p}\right)=1}\frac{1+\frac{1}{p^{s}}}{1-\frac{1}{p^{s}}}}$$
where $Q(x,y)$ is relatively prime to $2D$ and $k=2$ if $D<0$, while, if $D>0$, then $k=1$ and 
$$y>0, U\cdot(ax+by)>Ty$$
where $T,U$ are the smallest positive integer solutions of 
$T^2-DU^2=1.$  Finally, $x,y$ run over all pairs of positive integers.

\emph{Now to the proof that fills the gap.}

First suppose that $D$ is a\emph{ positive prime number $p$ of the form $4n+1$}.  Then 
\begin{equation}
\label{I1}
\sum\frac{1}{Q(x,y)^{s}}=\displaystyle\sum_{(n,2D)=1}\frac{1}{n^{2s}}\cdot\prod_{ \left(\frac{p}{A}\right)=1}\frac{1+\frac{1}{A^{s}}}{1-\frac{1}{A^{s}}}\prod_{ \left(\frac{p}{B}\right)=1}\frac{1+\frac{1}{B^{s}}}{1-\frac{1}{B^{s}}}
\end{equation}
where $A$ runs over all primes of the form $4n+1$ of which $p$ is a quadratic residue, and $B$ runs over all primes of the form $4n+3$ of which $p$ is a quadratic residue.

Secondly suppose that \emph{the same prime number $p=4n+1$ is now assumed to be a negative determinant $-p$ equal to $D$.}  Then
\begin{equation}
\label{I2}
\sum\frac{1}{Q_1(x,y)^{s}}=2\displaystyle\sum_{(n,2D)=1}\frac{1}{n^{2s}}\cdot\prod_{ \left(\frac{-p}{A_1}\right)=1}\frac{1+\frac{1}{A_1^{s}}}{1-\frac{1}{A_1^{s}}}\prod_{ \left(\frac{-p}{B_1}\right)=1}\frac{1+\frac{1}{B_1^{s}}}{1-\frac{1}{B_1^{s}}}
\end{equation}
where the sum on the left-hand side runs over all non-equivalent forms of determinant $-p$ , while $A_1$ and $B_1$ on the right hand side run over all primes of the form $4n+1$ and $4n+3$ respectively of which $-p$ is a quadratic residue.

However, it is well-known that for a prime number $A$ of the form $4n+1$ $-p$ is a quadratic residue or non residue of $A$ according as $+p$ is a quadratic residue or non residue of $A$ which implies that the set of primes $A$ \emph{coincides} with the set of primes $A_1$ whence
$$
\prod_{ \left(\frac{p}{A}\right)=1}\frac{1+\frac{1}{A^{s}}}{1-\frac{1}{A^{s}}}=\prod_{ \left(\frac{-p}{A_1}\right)=1}\frac{1+\frac{1}{A_1^{s}}}{1-\frac{1}{A_1^{s}}}.
$$  
However this is \emph{not} the case for the set of primes $B$ because here the Legendre symbol obeys the equation $\left(\frac{-p}{B}\right)=-\left(\frac{p}{B}\right)$.

\emph{We now prove Legendre's lemma by contradiction}.

Suppose that there exists \emph{no} prime number $B$ of which $p$ is a quadratic non residue.  Then $+p$ is also be a residue of all the prime numbers $B$, and therefore \emph{all prime numbers $B$ of the form $4n+3$} must appear in $\prod\frac{1+\frac{1}{B^{s}}}{1-\frac{1}{B^{s}}}$.  On the other hand, if since $-p$ is a non residue of all prime numbers of the form $4n+3$, then \emph{no prime} $B_1$ appears in  quantity $\prod\frac{1+\frac{1}{B_1^{s}}}{1-\frac{1}{B_1^{s}}}$ which means that\emph{ it has to collapse down to} $1$.

We now divide \eqref{I1} by \eqref{I2}. Then we see  the that the products involving $A$ and $A_1$ \emph{cancel}.  Moreover, if we multiply the numerator and denominator on the left hand side by $s-1$.
we obtain
\begin{equation}
\label{I3}
\boxed{\frac{\sum\frac{s-1}{Q(x,y)^{s}}}{\sum\frac{s-1}{Q_1(x,y)^{s}}}=\frac{1}{2}\prod\frac{1+\frac{1}{B^{s}}}{1-\frac{1}{B^{s}}}}
\end{equation}
where the primes $B$ are all primes of the form $4n+3$.

\emph{This equation is the goal of all of the previous developments.}

We now take the limit as $s\rightarrow 1$.  \emph{\textbf{Teege's identity }} shows that the limit of the left hand side is the \emph{finite number}
$$=\frac{h(p)\ln(T+U\sqrt{p})}{h(-p)\cdot 2\pi}$$
while the limit of the right-hand side is the infinite product

\begin{eqnarray*}
\frac{1}{2}\prod\frac{1+\frac{1}{B}}{1-\frac{1}{B}} & > & \frac{1}{2}\prod\frac{1}{1-\frac{1}{B}} \\
 & = &\frac{1}{2}\prod\left(1+\frac{1}{B}+\frac{1}{B^2}+\frac{1}{B^3}+\cdots\right)\\ 
 &=& \frac{1}{2}\sum\frac{1}{\prod B}\\
 &>&\frac{1}{2}\sum\frac{1}{B}\\
 &=&\frac{1}{2}\left(\frac{1}{3}+\frac{1}{7}+\frac{1}{11}+\frac{1}{19}+\cdots\right)
\end{eqnarray*}
where the denominators are all the primes of the form $4n+3$ and it is well known that the series \textbf{\emph{diverges}}
This contradiction proves Legendre's Lemma.
$\hfill \blacksquare$
\\


\section{Teege's refinement}

We will prove that the prime $B$ of the form $4n+3$ whose existence is guaranteed by Legendre's Lemma \emph{can be taken to be smaller than} $p$.  This latter condition was never mentioned by Legendre and is original with Teege.
\begin{proof}
We  give an expanded arrangement of Teege's own proof which uses infinite descent.

By Legendre's lemma there exists a prime number $B$ such that the Legendre symbol $\left(\dfrac{-p}{B}\right)=+1$, which means and  there also exists an equation
\begin{equation}
\label{Teq}
x^2+p=B\cdot b'.
\end{equation}
for some integer $b'$.

First, if $B<p$, then we are done.

Second, if $B>p$, then \emph{we claim:}
$$
b'<B.
$$
In fact,
$$
x<B\Rightarrow x^2+p=Bb'<B^2+p\Rightarrow b'<B+\frac{p}{B}<B+1\Rightarrow b'\leq B.
$$
so either $b'=B$ or we have the strict inequality $b'<B$.  But, if
$$
b'=B\Rightarrow x^2+p=B^2\Rightarrow p=(B-x)(B+x)\Rightarrow B-x=1, B+x=p
$$
and this means
$$
 B=\frac{p+1}{2}<p
$$
since $p>1$ and this is a contradiction.  Thus we \emph{cannot} have $b'=B$.  This proves our claim.
$\hfill \blacksquare$
\\
\\
\indent Now we will apply Fermat's infinite descent.

 Since  $x$ is even and $p$ is of the form $8n+1$, $b'$  is of the form $4n+3$, and so $b'$  must have a prime factor $B'$ of the form $4n+3$, whether or not $p$ and $b'$ have a common factor $p$, such that $\left(\dfrac{-p}{B'}\right)=+1$.  Thus there exists an equation

$$x_1^2+p=B'\cdot b''$$

\noindent where $B'<B$ and $b''$ is again smaller than $B'$ and of the form $4n+3$, from which we can in the same way deduce that there exists a third prime number $B''$ of the form $4n+3$ and smaller than $B'$ for which $\left(\dfrac{-p}{B''}\right)=+1$. By repeating this process we obtain a sequence of decreasing prime numbers
$$
B>B'>B''>\cdots
$$
 of the form $4n+3$ of which $p$ is a nonresidue and finally one smaller than $p$.  Therefore we have proven Gauss' lemma that \emph{there always exists a prime number $q<p$ for which a given prime number $p=8n+1$ a quadratic nonresidue }but with the refinement that at \emph{least one of these prime numbers $q$ must have the form $4n+3$.}
\end{proof}

We note that a very brief numerical search showed \emph{no} example in which $B>p$.


\section{Rogers' proof of Legendre's Lemma.}

We now add some remarks on Rogers' paper \cite{Rogers}.  He rectifies Legendre's original proof (1785) by using the following lemma of Selberg \cite{Selberg}.
\begin{thm}
If $D$ is a nonsquare integer, then
$$
\sum_{p\leq x, \left(\frac{D}{p}\right)=1}\frac{\ln p}{p}=\frac{1}{2}\ln x +O(1),
$$
where $p$ runs over prime values only.
\end{thm}
$\hfill \blacksquare$

Although Rogers never addresses Legendre's lemma directly, he uses Selberg's result and some elegant mathematics to single out the contribution of distinct arithmetic progressions of primes to Selberg's total sum and in particular to prove the following theorem.

\begin{thm}{\label{Slbg}}
Let $D$ be a nonsquare. Then 
$$
\sum_{p\leq x,p=4n+3, \left(\frac{D}{p}\right)=-1}\frac{\ln p}{p}=\frac{1}{4}\ln x +O(1),
$$
\end{thm}

\begin{proof}
We sketch Rogers' nice proof.
Let
$$
s_i:=\sum_{p\leq x, p=4n+i}\frac{\ln p}{p}, \quad \text{for}\quad i=1,3,
$$
and let $s_i^{+},s_i^{-}$ denote the corresponding sums over $p$ for which $\left(\frac{D}{p}\right)=\pm 1$, respectively.  First we note that Selberg's lemma and the Mertens' classical result
$$
\sum_{p\leq x}\frac{\ln p}{p}=\ln x+O(1)
$$
together imply
$$
\sum_{p\leq x, \left(\frac{D}{p}\right)=-1}\frac{\ln p}{p}=\frac{1}{2}\ln x +O(1).
$$
Now
\begin{eqnarray*}
s_1^{+}+s_3^{+} & = & \sum_{p\leq x, p=4n+1,\left(\frac{D}{p}\right)=1}\frac{\ln p}{p}+\sum_{p\leq x, p=4n+3,\left(\frac{D}{p}\right)=1}\frac{\ln p}{p} \\
 & = & \sum_{p\leq x, \left(\frac{D}{p}\right)=1}\frac{\ln p}{p}=\frac{1}{2}\ln x +O(1)
\end{eqnarray*} by Selberg's lemma.  A similar computation in which replacing $D$ by $-D$ leaves $s_1^{\pm}$ invariant while it interchanges $s_3^{+}$ and $s_3^{-}$ and shows
$$
s_1^{-}+s_3^{-} =s_1^{+}+s_3^{-} =s_1^{-}+s_3^{+} =\frac{1}{2}\ln x +O(1)
$$
Thus we have four linear equations in the four unknowns $s_i^{\pm}$, and solving them we obtain
$$
s_i^{\pm}=\frac{1}{4}\ln x +O(1).  
$$
The case $s_3^{-}$ is the equation stated in the theorem

\end{proof}

\noindent Now we let $x\rightarrow\infty$ in Theorem~\ref{Slbg}.  This proves the following result.
\begin{cor}
There exist infinitely many primes $p\equiv 3 \pmod 4$ for which $\left(\frac{D}{p}\right)=-1$
\end{cor}

In particular, if we put $D$ equal to a prime of the form $4n+1$ we reach our goal.

$\hfill \blacksquare$
\begin{cor}
For every prime $a$ of the form $4n+1$ there exist infinitely may primes $b$ of the form $4n+3$ for which $\left(\frac{a}{b}\right)=-1$
\end{cor}
$\hfill \blacksquare$

Of course, this is (a strong form of ) Legendre's Lemma.

Rogers actually proves a much stronger result which we do not need; we have extracted only a special case from the first part of his investigation.

It would seem that Rogers' paper has suffered the same fate as Teege's.  Although it appeared in a distinguished journal, the Mathematical Reviews and Zentralblatt fur Mathematik und ihre Grenzgebiete  only reproduce the introduction in his paper, while there is no comment on the mathematical content and, with the exception of Lemmermeyer's book \cite{Lemmermeyer} we have not found any references to his paper in the almost half-century since its publication in any of the papers/textbooks that deal with the first proof of the reciprocity theorem.

We remark that if one writes out in full detail the proof of Selberg's lemma, it is almost as long and complicated as the proof of Dirichlet's class number formulas.  Nevertheless, Rogers refers to it as ``...a simple lemma..."

Finally we note that Selberg used his identity to prove the non-vanishing of an L-series just as Dirichlet used his class-number formulas.  But, up to now, the latter have had an enormous influence on the subsequent development of algebraic number theory and commutative algebra while the former, as of yet, has not.

\section{Retrospect}

Apparently Teege wrote his first attempted proof of Legendre's Lemma \cite{Teege1} before 1914, although it was not published until 1920, perhaps because of the intervention of WW I.  But it contained an important error which he repaired in his second attempt \cite{Teege} which was published four years later.  This second paper also contained the first appearance of Teege's refinement (as he, himself, writes with pride.)  The magazine's editor writes in a footnote that the manuscript of Teege's second paper was ready for printing although the author was deceased. (The paper appeared with a dagger, $\dagger $, along side the author's name).  Then, as noted earlier, unacountably his paper was later forgotten.  Given Teege's identity (which Teege never explicitly sets out but is an easy consequence of his reasoning: ``...by a stroke of the pen" as G. Watson wrote of an implicit identity of Ramanujan) his proof of Legendre's lemma is quite elegant and deserves to be known.

We see that the proof of Teege's identity, itself, is long and complex because it depends on (the easy part (!)) of Dirichlet's class number formulas and the proof of the latter is quite involved.  Nevertheless, we are compensated by the fact that its use in the proof of Legendre's Lemma  \emph{is} elementary and quite beautiful and elevates Legendre's attempted proof to the list of complete proofs of the quadratic reciprocity law.


\end{document}